\documentclass[11pt]{article}

\usepackage[letterpaper,margin=1in]{geometry}
\usepackage{setspace}

\usepackage[T1]{fontenc}
\usepackage{baskervillef}
\usepackage[varqu,varl,var0]{inconsolata}
\usepackage[scale=.95,type1]{cabin}
\usepackage[baskerville,vvarbb]{newtxmath}
\usepackage[cal=boondoxo]{mathalfa}
\usepackage{latexsym}

\usepackage{amssymb}

\usepackage{geometry}
\usepackage{graphicx}
\usepackage{epstopdf}
\usepackage{enumerate}
\usepackage{xcolor}
\usepackage{amsthm}
\usepackage{amscd}
\usepackage{amsmath}
\usepackage{thm-restate}
\usepackage{subfiles}
\usepackage{hyperref}
\usepackage{tikz}
\usepackage{subcaption}
\usepackage{utf8math}
\usepackage{todonotes}
\usepackage{graphicx}

\colorlet{cyantransparent}{cyan!50} 
\colorlet{pinktransparent}{pink!50} 
\colorlet{greentransparent}{green!45}
\colorlet{bluetransparent}{blue!45}

\colorlet{cyantransparent}{cyan!50} 
\colorlet{pinktransparent}{pink!50} 
\colorlet{greentransparent}{green!45}

\usetikzlibrary{arrows.meta,decorations.markings}

\usetikzlibrary{arrows.meta,decorations.markings}

\newcommand{\R}{\mathbb{R}}

\newcommand{\Z}{\mathbb{Z}}
\newcommand{\NP}{{\sf NP}}
\newcommand{\one}{\mathbf{1}}
\newcommand{\zero}{\mathbf{0}}

\newcommand{\floor}[1]{\left\lfloor #1 \right\rfloor}
\newcommand{\ceil}[1]{\left\lceil #1 \right\rceil}

\DeclareMathOperator{\supp}{supp}

\DeclareMathOperator{\forcing}{forcing}
\DeclareMathOperator{\forcingset}{forcing-set}
\DeclareMathOperator{\xc}{xc}

\newcommand{\eps}{\varepsilon} 
\newcommand{\epsbis}{\varepsilon'}

\newcommand{\numrounds}{t}
\newcommand{\rhs}{\delta}

\newtheorem{theorem}{Theorem}
\newtheorem{lemma}[theorem]{Lemma}
\newtheorem{proposition}[theorem]{Proposition}

\newcommand{\rb}[1]{\left(#1\right)}
\newcommand{\set}[1]{\left\{#1\right\}}
\newcommand{\norm}[1]{\left\lVert#1\right\rVert}

\renewcommand{\epsilon}{\varepsilon}

\theoremstyle{remark}

\usepackage[
    style=numeric,
    natbib=true,      
    doi=false,
    url=false,
    isbn=false,
    eprint=false,
    backend=bibtex,
    maxbibnames=99,
    minbibnames=99
]{biblatex}

\title{Polyhedral extended formulations that approximate the Gomory closure for packing problems} 
\author{
Friedrich Eisenbrand \thanks{EPFL, Lausanne, Switzerland, \texttt{\{friedrich.eisenbrand, jiaye.wei\}@epfl.ch}}
\and
Samuel Fiorini \thanks{Université libre de Bruxelles, Brussels, Belgium, \texttt{Samuel.Fiorini@ulb.be}}
\and
Lars Rohwedder \thanks{University of Southern Denmark, Odense, Denmark, \texttt{rohwedder@sdu.dk}}
\and
Jiaye Wei \footnotemark[1]
}

\date{} 

\begin{document}
\maketitle

\begin{abstract}

  We consider $0/1$ \emph{packing problems} of the form $\max\{w^\intercal x ：Ax≤\one, \, \allowbreak  x ∈ \{0,1\}^n\}$,
  with $A ∈ ℝ^{m ×n}$ nonnegative, 
  together with their continuous relaxation
  $P  = \{x \in \R^n : Ax≤\one, \, x ∈ [0,1]^n\}$.  The \emph{Gomory
    closure} $P'$ of $P$ is the intersection of $P$ with
  all its Gomory cutting planes.  In many cases of interest, the
  Gomory closure yields a significantly tighter relaxation than
  $P$. 
  At
  the same time, 
  optimizing over $P'$ is \NP-hard.

  Our main result is a \emph{polyhedral} relaxation of
  polynomial extension complexity that approximates $P'$ up to a
  $1+\epsilon$ factor for any fixed $\epsilon > 0$. This extends to
  the $\numrounds$-th iterate $P^{(\numrounds)}$ of the Gomory closure
  for any fixed $\numrounds$. Prior to our work, 
  Mastrolilli (2020) obtained a corresponding \emph{non-polyhedral} relaxation from the Lasserre hierarchy. 

     
  We introduce the notion of a
  \emph{locally-near-exact} relaxation, which, as we show,
  characterizes any fixed-rank Gomory closure up to a $1+ε$ factor: we
  prove that, for an appropriate choice of parameters, every
  locally-near-exact relaxation
  approximates $P^{(\numrounds)}$ up to a $1+ε$ factor. We also prove a converse to
  this: when $\numrounds$ is large enough, $P^{(\numrounds)}$ provides
  a locally-near-exact relaxation.
\end{abstract}

\section{Introduction} \label{sec:introduction}
An \emph{integer program} is a problem of the form
\begin{equation}
  \label{eq:01-ip}
  \max \{w^\intercal x ：Ax ≤ b, x ∈ ℤ^n\},  
\end{equation}
with $w ∈ \R^n$, $A ∈ \R^{m ×n}$ and $b ∈ \R^m$.
This framework captures many important \NP-hard optimization problems, see, e.g.~\cite{Schrijver03,wolsey2020integer}. Successful approaches to tackle~(\ref{eq:01-ip}) are largely based on the solution of the \emph{linear relaxation} 
  $\max \{w^\intercal x ：Ax ≤ b, x ∈ ℝ^n\}$. 
  The linear relaxation can be solved efficiently and it provides upper bounds, also for sub-problems of~(\ref{eq:01-ip}) that then help in pruning the search space. 
  The efficiency of this approach depends on the quality of the linear relaxation.

  \emph{Cutting planes} provide a way to strengthen any  relaxation of~\eqref{eq:01-ip}. 
  The principle was invented by Gomory~\cite{Gomory58} who made the following observation.  
Denote the polyhedron of feasible solutions of the linear relaxation by $P$. In other words, let $P := \{x \in \R^n : Ax \le b\}$. If $c ∈ℤ^n$ is an integer vector and the inequality $c^\intercal x ≤ δ$ is satisfied by each $x ∈ P$ , then
\begin{equation}
  \label{eq:3}
  c^\intercal x ≤ ⌊ δ ⌋
\end{equation}
is satisfied by each feasible \emph{integral}  $x ∈ P ∩ℤ^n$. Based on this, Chvátal~\cite{Chvatal73} invented the corresponding concepts of closure and hierarchy. 
The \emph{first closure} $P'$  of $P$ is the intersection of all Gomory cutting planes that can be derived for $P$. 
Iterating this procedure $i$~times from $P$ yields a (rational) polyhedron denoted by $P^{(i)}$ for each $i \in \Z_{\geq 0}$. One obtains a \emph{hierarchy} $P ⊇ P' ⊇ \cdots ⊇ P^{(i)} ⊇ \cdots ⊇ P_I$ where $P_I$ is the convex hull of all integral solutions, also known as the \emph{integer hull} of $P$. It is known that this hierarchy is finite, that is, there is an integer $k \in \Z_{\geq 0}$ such that $P^{(k)} = P_I$~\cite{MR597387}. The smallest $k ∈ \Z_{\geq 0}$ satisfying this is referred to as the \emph{Chvátal rank} or simply \emph{rank} of $P$.

For relaxations $P$ that are contained in $[0,1]^n$, other important hierarchies have been developed since then, such as \emph{lift-and-project}~\cite{BCC96}, the hierarchies of Sherali and Adams~\cite{SheraliAdams90}, Lov{\'a}sz and Schrijver~\cite{LovaszSchrijver90} as well as the one of Lasserre~\cite{MR2041934}. The last two in this list contain relaxations that are non-polyhedral and based on \emph{semidefinite programming}. 
Common to all these specific hierarchies for the 0/1-cube are two important features: (i) optimization over the $t$-th level of the hierarchy can be done in polynomial time for fixed $t$, (ii) the corresponding rank of a polytope $P ⊆ [0,1]^n$ is bounded by $n$~\cite{MR1997246}. 

 An integer program~\eqref{eq:01-ip} is a \emph{packing problem} if the constraints are of the form $Ax ≤ \one, \, x ∈ \{0,1\}^n$ with $A ∈ \R_{≥0}^{m ×n}$. 
 In this case, the feasible set of the linear relaxation is a polytope $P = \{x \in [0,1]^n :  Ax ≤ \one \}$. 
 We call such a $P$ a \emph{packing polytope}. Observe that $P^{(\numrounds)}$ is a packing polytope as well for every $t \in \Z_{\ge 0}$. Gomory cutting planes are particularly strong for packing problems such as mat\-ching problems in graphs or hypergraphs~\cite{cl12,st10}.
 The first Gomory closure of the fractional matching polytope is already its integer hull~\cite{Edmonds65b}, while the rank of the fractional matching polytope is linear with respect to the other hierarchies mentioned above~\cite{au2016comprehensive,mathieu2009sherali}. 
Optimizing over the first Gomory closure is \NP-hard~\cite{eisen99}, even for polytopes in the cube~\cite{cornuejols2020rational,cornuejols2018gomory} and this extends to packing polytopes, as we show in Appendix~\ref{sec:hardness}. The strength and intractability of the Gomory closure motivates the study of its approximation. 

An \emph{extended formulation} of a convex set $Q ⊆ ℝ^n$ relative to a cone $K ⊆ ℝ^k$ is a description of $Q$ as the image under an affine map of an affine slice of $K$. In other words, it is a description of the form $Q = \{x \in ℝ^n : \exists y \in K : Cy = d,\ x - Ey = f\}$ where $C \in ℝ^{\ell \times k}$, $d \in ℝ^\ell$, $E \in ℝ^{n \times k}$, $f \in ℝ^n$. An extended formulation relative to the nonnegative orthant $ℝ^k_{\ge 0}$ is called a \emph{size-$k$ linear extended formulation}. Assuming $k = {t+1 \choose 2}$ and $K = \mathbb{S}^t_+$ (the cone of all $t \times t$ symmetric positive semidefinite matrices), an extended formulation relative to $K$ is called a \emph{size-$t$ positive semidefinite extended formulation}. The \emph{linear (resp.\ positive semidefinite) extension complexity} of $Q ⊆ ℝ^n$ is defined as the minimum size of a linear (resp.\ positive semidefinite) extended formulation for $Q$, or $+\infty$ if no such extended formulation exists. The convex sets of \emph{finite} linear (resp.\ positive semidefinite) extension complexity are the polyhedra (resp.\ spectrahedral shadows). Unless stated otherwise, \emph{extension complexity} means linear extension complexity. The extension complexity of a polyhedron $Q$ is usually denoted by $\xc(Q)$.

%
For $\alpha\ge1$, we call a convex set $Q\subseteq\R^n_{\ge0}$ an \emph{$\alpha$-approximation} of $P^{(t)}$ if $P_I\subseteq Q\subseteq\alpha P^{(t)}$. Throughout the paper, a relaxation of $P_I$ is understood to be a convex set in $\R^n_{\ge0}$ containing $P_I$.

Mastrolilli~\cite{mas20} provided a spectrahedral shadow $Q \subseteq \R^n_{\ge 0}$ that approximates the $t$-th Gomory closure of any packing polytope. More precisely, for $ε > 0$ and $t ∈ \mathbb{N}$, Mastrolilli constructs such a $Q$ with $P_I ⊆ Q ⊆ (1+ε) P^{(t)}$. Here $P_I$ is the \emph{integer hull} of $P$. His relaxation is defined via the $O(t/ε)$-th round of the Lasserre hierarchy, which gives a size-$(mn)^{O(t/ε)}$ positive semidefinite extended formulation for $Q$.

\paragraph*{Our contributions.} Our \emph{main result} is a polyhedral relaxation approximating $P^{(t)}$ with a polynomial-size linear extension complexity.
 
\begin{restatable}{theorem}{approxcg}\label{thm:main-1}
Let $P = \{x \in [0,1]^n : Ax \le \one\}$, where $A \in \R_{\geq 0}^{m \times n}$. For each $\eps \in (0,1)$ and $\numrounds \in \Z_{\ge1}$, there exists a polyhedral relaxation $Q$ of $P_I$ that is a $(1+\eps)$-approximation of the $\numrounds$-th Gomory closure $P^{(t)}$ of $P$ and has extension complexity
\[
      (mn)^{O(2^t/\eps \cdot \log(1/\eps))}\,.
\]
\end{restatable}
For the \emph{first closure}, that is, for $t=1$, the extension complexity of our polyhedral relaxation is comparable to that of Mastrolilli's non-polyhedral relaxation, up to a $\log (1/ε)$-factor in the exponent.  We note that the first closure plays an important role, also in practical IP solving~\cite{fischetti2007optimizing}. From a practical angle, semidefinite relaxations are still challenging today.


\paragraph*{On further improvements.} Extending the celebrated lower bound of Rothvoss~\cite{rothvoss2017matching}, Sinha~\cite{sin18} proved, for each fixed $\eps\in(0,1)$, an $n^{\Omega(1/\eps)}$ lower bound on the size of any linear extended formulation providing a $(1+\eps)$-approximation of the integer hull of the matching polytope. As the first closure of the fractional matching polytope is integral, this shows that 
our construction for the first closure is optimal up to the $\log (1/ε)$-factor in the exponent.

\paragraph*{Characterization of the approximate Gomory closure.}

Towards Theorem~\ref{thm:main-1}, we develop the notion of a \emph{locally-near-exact} relaxation of $P_I$. We prove that it quantifies what is captured by the fixed-rank Gomory closure. The notion may be of independent interest. A simple linear extended formulation inspired by the Dantzig-Wolfe decomposition yields a locally-exact relaxation and thereby proves Theorem~\ref{thm:main-1}.

Let us now explain the property mentioned above. 
Fix a locality parameter $r \in \Z_{\ge 1}$ and consider a selection of $r$ rows $T = (i_1,i_2,\ldots,i_r) \in [m]^r$, allowing repetitions. We define the \emph{neighborhood} of $T$ as
\begin{equation}\label{eq:neighborhood}
	N(T) := \left\{j\in[n]:\sum_{\ell=1}^r A_{i_\ell j} \ge 1 \right\}. 
\end{equation}
In the case that $A\in\{0,1\}^{m\times n}$, the neighborhood corresponds simply to all columns that have at least one non-zero entry in $T$. In other words, in this case, $N(T)$ is the neighborhood of $T$ (regarded this time as a set) computed in the bipartite graph with $m+n$ vertices in which the vertex for the $i$-th row is adjacent to the vertex for the $j$-th column whenever $A_{ij} = 1$.

Another example, is the knapsack polytope $P = \{ x ∈ [0,1]^n ： a^\intercal x ≤ 1\}$ with $a\ge\zero$. For $T = (1,\dots,1)∈ [m]^r$ the neighborhood  $N(T)$ consists of all variable indices $j$ that satisfy $a_j ≥ 1/r$. 

Let $Q$ be a relaxation of $P_I$ and let $\eps\ge0$. We say that $Q$ is \emph{$(r, \eps)$-locally-near-exact}\footnote{Replacing the $(1+\epsilon)$ factor in the definition with any $\alpha \ge 1$, one can define the general notion of a \emph{$r$-locally-$\alpha$-approximate} relaxation. We will not need this notion, but it could be useful in other contexts.} if for every $T \in [m]^r$ and every inequality $\sum_{j\in N(T)} c_j x_j \le \rhs$ with nonnegative coefficients that is valid for $P_I$, the relaxed inequality $\sum_{j\in N(T)} c_j x_j \le (1 + \epsilon) \rhs$ is valid for $Q$. We say that a relaxation is \emph{$r$-locally-exact} if it satisfies the previous condition with $\eps = 0$.
For example, consider again a  knapsack polytope $P=\{x\in[0,1]^n:a^\intercal x\le1\}$ with $a\ge\zero$. 
A $r$-locally-exact relaxation captures every valid inequality for $P_I$ with nonnegative coefficients supported on the items of weight at least $1/r$.

Theorems~\ref{thm:local_to_CG} and \ref{thm:CG_to_local} formalize the tight relationship between fixed-rank Gomory closures of packing polytopes and locally-near-exact relaxations.

\begin{restatable}{theorem}{local_to_CG}\label{thm:local_to_CG}
	Let $P = \{x \in [0,1]^n : Ax \le \one\}$ be a packing polytope, where $A \in \R_{\geq 0}^{m \times n}$. For each $\eps \in (0,1)$ and $\numrounds\in \Z_{\geq 1}$, there is some $r = O(2^t/\eps \cdot \log (1/\eps))$ such that every $(r,\epsilon/3)$-locally-near-exact relaxation of $P_I$ that is contained in $P$ provides a $(1+\eps)$-approximation of $P^{(\numrounds)}$.
\end{restatable}

We also prove the following converse of Theorem~\ref{thm:local_to_CG}.

\begin{restatable}{theorem}{converse}\label{thm:CG_to_local}
	Let $P = \{x \in [0,1]^n : Ax \le \one\}$ be a packing polytope, where $A \in \R_{\geq 0}^{m \times n}$. Let $r \in \Z_{\ge 1}$ and $\eps > 0$. Then there is some $t = O(r^2 + r\log(1/\eps))$ such that $P^{(t)}$ provides a $(r,\eps)$-locally-near-exact relaxation of $P_I$.
\end{restatable}

For fixed $r$, one can construct a polynomial-size linear extended formulation defining an $r$-locally-exact relaxation, see Lemma~\ref{lem:local} below. We provide some intuition as to why the lemma holds. Fix some $T = (i_1,i_2,\ldots,i_r) \in [m]^r$, and consider the corresponding neighborhood $N(T)$. Let $y \in P \cap \{0,1\}^n$ be such that $\supp(y) \subseteq N(T)$, where $\supp(y) := \{j \in [n] : y_j \neq 0\}$ denotes the support of $y$. Then, 
\begin{equation}\label{eq:support}
	\frac{1}{r}\,|{\supp(y)}|
 \le\sum_j \frac1r\left(\sum_{\ell=1}^r A_{i_\ell j}\right)y_j
 =\frac1r\sum_{\ell=1}^r \underbrace{\sum_j A_{i_\ell j} y_j}_{\le 1}
 \le1\,,
\end{equation}
hence the size of the support of $y$ is at most $r$, and there are  $O(n^{r})$ such integer points $y$. We enforce local exactness by writing the restriction $x|_{N(T)}$ of $x$ to $N(T)$ as a convex combination of these points, using nonnegative variables $\lambda_{T,y}$ with $\sum_y \lambda_{T,y} = 1$.

\paragraph{Communication complexity.}
Besides explicit constructions such as those we use in the results above,
communication complexity and communication protocols in particular provide an alternative way to prove the existence of extended formulations.
 Yannakakis' deterministic protocol to determine whether the clique of player Alice and the stable set of player Bob intersect~\cite{Yannakakis91} can be used to construct a linear extended formulation whose projection on the original variables is a relaxation of the stable set polytope satisfying every clique inequality, see also~\cite{rao2020communication}. We generalize Yannakakis' protocol to a protocol computing the slack of a given integer point of $P$ with respect to a given valid inequality for $P_I$ with right-hand side $O(\numrounds/\eps)$. Here, $P$ is an arbitrary packing polytope. Using ideas from Mastrolilli~\cite{mas20}, these inequalities suffice for a $(1+\eps)$-approximation of the $\numrounds$-th Gomory closure.

\begin{restatable}{theorem}{ef}\label{thm:main-2}
    Let $P = \{x \in [0,1]^n : Ax \le \mathbf 1\}$ be a packing polytope, where $A \in \R_{\geq 0}^{m \times n}$. For each $\eps \in (0,1)$ and $t \in \Z_{\ge 1}$, there exists a polyhedral relaxation $Q$ of $P_I$, which is a $(1+\eps)$-approximation of $P^{(t)}$ and has a linear extended formulation of size at most
    \[
      \xc(P) + n^{O(t/\eps \cdot \log n)}.
    \] 
\end{restatable}
This result offers a different tradeoff to Theorem~\ref{thm:main-1} that is particularly interesting if $m \gg n$. 
Note that the relaxation in Theorem~\ref{thm:main-2} can actually be constructed in quasi-polynomial time, by combining the protocol of Lemma~\ref{lem:small-rhs} with the construction of Aprile and Faenza~\cite[Theorem~5]{AprileFaenza20}. 
We present its proof in Appendix~\ref{sec:quasi-poly-construction}. 

\paragraph{Hypergraph matching.} For the maximum matching problem in $k$-uniform hypergraphs, building on previous work by Chan and Lau~\cite{cl12}, Singh and Talwar~\cite{st10} have shown that $O(k^2)$ rounds of Gomory cuts suffice to reduce the integrality gap of the standard linear relaxation to $(k+1)/2$. We improve this result to $O(\log k)$, see Theorem~\ref{thm:hyper} in Appendix~\ref{sec:hypergraph-matching}.
This gives further evidence for the strength of Gomory cuts in packing problems.
\paragraph*{Other related work.} 
%
Recall that a \emph{covering problem} is an integer program~\eqref{eq:01-ip} in which the linear constraints are of the form $Ax ≥ b, \zero \le x \le \one$ for some $A ∈ \R^{m×n}_{≥0}$ and $b ∈ \R^{m}_{\ge 0}$. 
Bienstock and Zuckerberg~\cite{bz06} provided a method to approximate any fixed-rank Gomory closure of a covering problem up to a factor of $1+ε$ in polynomial time, provided that $A \in \{0,1\}^{m \times n}$ and $b = \one$. 
The extended formulation given in~\cite{bz06}, providing a $(1+\eps)$-approximation of $P^{(\numrounds)}$, has size $(m+n)^{O(\numrounds^2/\eps^2)}$. 
Mastrolilli~\cite{mas20} provided a simpler construction based on ideas from polynomial optimization. Fiorini et al.~\cite{fhw21} further simplified the approach and improved the size bound to $(mn)^{O(\numrounds/ε)}$. We remark that the results above generalize, for instance, to the case where $A \in \{0,\ldots,\Delta\}^{m \times n}$ and $b \in \{0,\ldots,\Delta\}^m$ for some $\Delta = O(1)$. We refer the reader to the original papers for more details.

\subsubsection*{Warm up: Approximating $P'$.}
\label{sec:warm-up:-case}
We start by describing how to obtain our main result for  the first Gomory closure
$P'$ in the case where the packing polytope is \emph{defined by a
$0/1$-matrix} $A ∈ \{0,1\}^{m ×n}$.   This warm-up is self-contained and acquaints the reader with the main techniques and proposed concepts laid out above that then unfold to full generality in Section~\ref{sec:Gomory-from-locally-near-exact}. 
Recall that, for $ε ∈ (0,1)$, our goal is to find a polytope $Q ⊆ ℝ^n$
that satisfies $P_I ⊆ Q$ and $Q ⊆ (1+ε) P'$.

\medskip 

\noindent 
Let $r\in\Z_{\ge1}$. Let us first understand how to construct an extended formulation of an $r$-locally-exact relaxation. To this end, let $T\in[m]^r$ be a tuple of $r$ row indices. The key observation is that for the variable indices in the neighborhood $N(T)$, a point $x \in P$ can have at most $r$ coordinates equal to one. This is because there can be at most one component equal to~$1$ in the support of each row of $A$ and $N(T)$ is the union of these supports.  

\smallskip 
Let $A^{N(T)} \in \R^{m \times |N(T)|}$ be the restriction of $A$ to columns in $N(T)$ and consider the polytope $P_T = \{ x ∈ [0,1]^n ： A^{N(T)} x|_{N(T)} ≤ \one\}$. The observation above puts us in a position to formulate a size-$O(n^r)$ extended formulation of the integer hull $Q_T$ of $P_T$.\footnote{The reader will notice the similarity but crucial difference to \emph{Danzig-Wolfe decomposition}~\cite{dantzig1960decomposition}. The latter relies on the integer hull of the packing problem that is defined by the rows indexed by $T$, in which  $Q_T$ is generally strictly contained. } We can enumerate exactly the $0/1$-points $z_j ∈\{0,1\}^{|N(T)|}$ that can be extended to a feasible $0/1$-point of $P_T$. (There are at most $(n+1)^r = O(n^r)$ such points.) Outside of $N(T)$, a component can be set to $0$ or $1$ arbitrarily. The construction of an extended formulation of $Q_T$ is then straightforward, see the proof of Lemma~\ref{lem:local} for details.

We set
\begin{displaymath}
  Q = P\cap\displaystyle\bigcap_{T ∈ [m]^{r}} Q_T. 
\end{displaymath}
Clearly, $Q$ is a $r$-locally-exact relaxation contained in $P$: For an inequality $c^\intercal x ≤δ$ valid for $P_I$ and $T ∈ [m]^{r}$,  $\sum_{j ∈ N(T)} c_j x_j ≤ δ$ is already valid for $Q_T$ and thus for $Q$ as well. The size of this linear extended formulation is $O((mn)^r)$. 

\bigskip  
\noindent 
We now show that, for the choice $r = \left\lceil (1+1/ε) \cdot \log(1+1/ε) \right\rceil$, $Q$ approximates the first Gomory closure  $P'$. 
For this, we need to show that
\begin{equation}
  \label{eq:gom}
    c^\intercal x ≤ ⌊ δ ⌋ (1+ε) 
\end{equation}
is valid for $Q$ for  each  Gomory cut $c^\intercal x ≤ ⌊ δ ⌋$ of $P$.
Recall that a packing polyhedron is \emph{downward monotone}, meaning if $x≥y≥0$ and $x ∈ P$, then $y ∈P$ as well. 
We may restrict attention to nonzero $c\in\Z_{\ge0}^n$, in which case $\floor{\delta}\ge1$. 
\footnote{If some component of $c$ was negative, then the inequality is dominated by the valid inequality in which this component is set to zero. Also, $P$ contains each unit vector $e_i$ which implies $⌊δ⌋≥1$.} 
As observed by Mastrolilli~\cite{mas20}, see also~\cite{bs24}, if $\floor{\delta}\ge1/\eps$, then $\delta<\floor{\delta}+1\le(1+\eps)\floor{\delta}$, so~\eqref{eq:gom} is already valid for $P$, and hence for $Q$. We may therefore assume that $\floor{\delta}<1/\eps$, which implies $\delta<1+1/\eps$. It suffices to consider cuts with a \emph{derivation}, see, e.g.~\cite{CCPS98,Schrijver86}, of the form
\begin{equation}
\label{eq:lambda} 
  ⌊λ^\intercal A ⌋ = c^\intercal , \, λ^\intercal  \one = δ, \text{ with } λ ∈ \R_{\ge0}^m. \footnote{Here, the $ ⌊⋅⌋$-operator refers to \emph{rounding down} component-wise. It captures the effect of using the lower bounds $x_i ≥ 0$ in the derivation.} 
\end{equation}
We proceed with a probabilistic (averaging) argument to show validity of \eqref{eq:gom} for $Q$. 
The vector $p = λ / \|λ\|_1 ∈ ℝ_{≥0}^m$ defines a probability distribution of the rows of the constraint matrix $A$ with $p_i$ being the probability of selecting row $i$.

Fix a column $j\in\supp(c)$, and denote it by $A^j$. Since $A$ is a $0/1$-matrix, the probability that a row sampled according to $p$ contains $j$ in its support is
\[
    p^\intercal A^j
    =\frac{\lambda^\intercal A^j}{\delta}
    \ge\frac{c_j}{\delta}
    \ge\frac1\delta
    >\frac{\eps}{1+\eps}.
\]

Now sample $r$ row indices independently according to $p$, yielding a tuple $T=(i_1,\ldots,i_r)\in[m]^r$. For every $j\in\supp(c)$, the probability that none of these rows contains $j$ in its support satisfies
\[
 \Pr[j\notin N(T)]
 \le \left(1-\frac{\eps}{1+\eps}\right)^r
 \le \exp\!\left(-\frac{\eps r}{1+\eps}\right)
 \le \frac{\eps}{1+\eps}.
\]
Thus $\Pr[j\in N(T)]\ge 1/(1+\eps)$.
Since $c\ge\zero$, the inequality
$\sum_{j\in N(T)}c_jx_j\le\floor{\delta}$ is valid for $P_I$, and hence for the $r$-locally-exact relaxation $Q$.
For every $x\in Q$, averaging over $T$ gives
\[
 \frac{1}{1+\eps}c^\intercal x
 \le \sum_{j=1}^n\Pr[j\in N(T)]c_jx_j
 =\mathbb{E}_T\!\left[\sum_{j\in N(T)}c_jx_j\right]
 \le\floor{\delta}.
\]
This proves~\eqref{eq:gom}, and hence Theorem~\ref{thm:main-1} for the first closure $P'$ when $A\in\{0,1\}^{m\times n}$.
\begin{proposition}
  \label{prop:1}
  Let $P = \{[0,1]^n : Ax \le \one\}$, where $A \in \{0,1\}^{m \times n}$ has no zero column. For each $\eps \in (0,1)$  there exists a polyhedral relaxation $Q$ of $P_I$ that is a $(1+\eps)$-approximation of the  Gomory closure $P'$ of $P$ and has extension complexity
\[
      (mn)^{O(1/\eps \cdot \log(1/\eps))}\,.
\]
\end{proposition}

\section{Approximate Gomory closures from locally-near-exact relaxations}\label{sec:Gomory-from-locally-near-exact}

This section is devoted to proving Theorems~\ref{thm:main-1} and \ref{thm:local_to_CG} in full generality. 
We first prove that we can obtain locally-exact relaxations via polynomial-size linear extended formulations.

\begin{lemma}\label{lem:local}
	Let $r \in \Z_{\ge 1}$. There exists a downward monotone, $r$-locally-exact relaxation $Q$ such that $P_I \subseteq Q \subseteq P$, with a linear extended formulation of size at most $O((mn)^r)$.
\end{lemma}

\begin{proof}
For each $T\in[m]^r$, let $\mathcal{Y}_T := \{y\in P\cap\{0,1\}^n:\supp(y)\subseteq N(T)\}$. Consider the following constraints in variables $x_j$ (where $j \in [n]$) and $\lambda_{T,y}$ (where $T \in [m]^r$ and $y \in \mathcal{Y}_T$):
\[
\begin{aligned}
    \sum_{y\in\mathcal{Y}_T}y_j\lambda_{T,y}&=x_j
        &&\text{for }T\in[m]^r,\ j\in N(T),\\
    \sum_{y\in\mathcal{Y}_T}\lambda_{T,y}&=1
        &&\text{for }T\in[m]^r,\\
    \lambda_{T,y}&\geq 0
        &&\text{for }T\in[m]^r,\ y\in\mathcal{Y}_T\,.
\end{aligned}
\]
Let $Q$ denote the set of points $x \in P$ such that there exists a $\lambda$ satisfying the above constraints. It is easy to see that every $x \in P \cap \{0,1\}^n$ belongs to $Q$, since whenever $T \in [m]^r$ is fixed, the point $y \in \{0,1\}^n$ with $y_j := x_j$ for $j \in N(T)$ and $y_j := 0$ otherwise belongs to $\mathcal{Y}_T$ and is trivially a convex combination of points in $\mathcal{Y}_T$. For each $T$, the constraints express $x|_{N(T)}$ as a convex combination of restrictions of integer points of $P$, so $Q$ is $r$-locally-exact. Since each $\mathcal{Y}_T$ is downward monotone (as a subset of $\{0,1\}^n$), $Q$ is downward monotone (as a subset of $[0,1]^n$).

Fix some $T \in [m]^r$. By \eqref{eq:support}, we see that $\supp(y)$ has size at most $r$, for every $y \in \mathcal{Y}_T$. It follows that $\mathcal{Y}_T$ has size at most $(n+1)^{r}$. Since there are $m^r$ choices for $T$, together with $Ax \le \one$ and $\zero \le x \le \one$, we obtain an extended formulation defined by at most $m + 2n + m^r (n+1)^{r}$ inequalities, implying the stated size upper bound.
\end{proof}

We now set up to prove Theorem~\ref{thm:local_to_CG}. 
We may always assume $A \in [0,1]^{m \times n}$ or equivalently that $P$ contains the standard simplex.
Suppose that $A_{ij} > 1$ for some $i \in [m]$ and $j \in [n]$. This forces $x_j=0$ for every $x \in P^{(\numrounds)}$ (since $\numrounds \ge 1$), and also for every $x \in Q$ whenever $Q$ is a locally-near-exact relaxation. This implies that we may safely delete all coordinates $x_j$ for which $A_{ij} > 1$ for at least one $i \in [m]$, apply the theorem under consideration (that is, Theorem~\ref{thm:main-1} or \ref{thm:local_to_CG}) with a smaller value of $n$, and deduce the theorem for the original value of $n$. In the event that all coordinates are forced to zero, one has $P^{(\numrounds)} = P_I = \{\zero\}$ and the theorem trivially holds. 

First, we observe that $P$ itself gives a $2^i$-approximation of $P^{(i)}$ for all $i \in \Z_{\ge 0}$.

\begin{lemma}\label{lem:scaling}
For every packing polytope $P \subseteq [0,1]^n$ containing the standard simplex, and for every $i \in \Z_{\ge 0}$, $P$ is a $2^i$-approximation of $P^{(i)}$.
\end{lemma}

\begin{proof}
It suffices to show the result for $i = 1$ since $P'$ contains the standard simplex whenever $P$ does. Consider any inequality $c^\intercal x \le \delta$ that is valid for $P$ and has $c \in \mathbb{Z}_{\geq 0}^n\setminus \{\zero\}$. Observe that $\delta \ge 1$ since $c$ is nonzero and $P$ contains the standard simplex. Hence, $\delta \le 2\floor{\delta}$ and $c^\intercal x \le 2 \floor{\delta}$ is valid for $P$. The lemma follows.
\end{proof}

Next, we prove that locally-near-exact relaxations approximately capture all valid inequalities $c^\intercal x \leq \delta$ for $P_I$ with nonnegative integer coefficients such that $\max \{c^\intercal x : x \in P\}$ is small. 

\begin{lemma}\label{lem:sampling}
Let $Q$ be a $(r,\eps)$-locally-near-exact relaxation of $P_I$. Suppose that $c^\intercal x\leq\delta$ is valid for $P_I$, where $c\in\mathbb{Z}_{\geq 0}^n\setminus\{\zero\}$ and $\delta\in\mathbb{Z}_{\geq 1}$.
Then for each $\gamma>0$ and $\eta \in (0,1)$, the inequality $c^\intercal x\leq(1+\eta)(1+\eps)\delta$ is valid for $Q$ if
	\begin{enumerate}
		\item $c^\intercal x\leq \gamma$ is valid for $P$, and
		\item $r \geq 8\gamma\log(1/\eta+1)$. 
	\end{enumerate}
\end{lemma}

\begin{proof}
By linear programming duality, since $c^\intercal x\leq \gamma$ is valid for $P$, there exists
$\lambda\in\mathbb{R}_{\geq 0}^m$ with
$\lambda^\intercal A\geq c^\intercal$ and
$0<\lambda^\intercal\mathbf{1}\leq \gamma$.

For each $i \in [m]$, define $p_i := \lambda_i/(\lambda^\intercal\mathbf{1})$.
By integrality of $c$, we have $c_j\geq 1$ for each $j\in\supp(c)$
and therefore $\sum_i p_iA_{ij}\geq 1/\gamma$.
Sample $i_1,\ldots,i_r \in [m]$ independently according to $p$ and let
$T := (i_1,\ldots,i_r)$. For each $j \in \supp(c)$, the random variables $A_{i_1 j}$, \ldots, $A_{i_r j}$ all lie in $[0,1]$, and their sum is in expectation at least $r/\gamma$. 
By the choice of $r$, we have $r/(2\gamma)>1$. Applying the Chernoff bound, we get
\begin{equation*}
    \Pr[j\notin N(T)] = \Pr[\sum_{\ell=1}^r A_{i_\ell j} < 1] 
    \le \Pr[\sum_{\ell=1}^r A_{i_\ell j} \le \frac{r}{2\gamma}]
    \le e^{-r/(8\gamma)} \le \eta/(1+\eta).
\end{equation*}
For every $x\in Q$, we obtain
\begin{equation*}
    (1+ \eta)^{-1} c^\intercal x
    \leq\sum_{j}\Pr[j\in N(T)]c_jx_j
    =\mathbb{E}\!\left[\sum_{j\in N(T)}c_jx_j\right]
    \leq (1+\eps) \delta\,,
\end{equation*}
where the last inequality follows from the fact that $Q$ is 
$(r,\eps)$-locally-near-exact. 
\end{proof}
We are now ready to prove Theorem~\ref{thm:local_to_CG} by induction on the rank. Cuts with a small right-hand side are handled by Lemma~\ref{lem:sampling}. For the remaining cuts, the rounding loss is absorbed by the induction step.

\begin{proof}[Proof of Theorem~\ref{thm:local_to_CG}]
Let $\epsbis \in (0,1)$ be a scaled variant of $\eps$ that we will specify later and define
\[
    \gamma := \frac{16\cdot2^t}{\epsbis},
    \qquad
    r := \left\lceil 8\gamma\log(1+2/\epsbis)\right\rceil .
\]
Let $Q$ be an $(r, \epsbis)$-locally-near-exact relaxation of $P_I$ that is contained in $P$.
For $i \in \{0,\ldots,\numrounds\}$, let $\alpha_i := (1+\epsbis/2+2^{i+3}/\gamma)(1+\epsbis)$. We prove by induction on $i$ that $Q$ is an $\alpha_i$-approximation of $P^{(i)}$. For $i=0$, this follows directly from $Q\subseteq P$. 
Suppose that $1\leq i\leq t$ and that the claim holds for $i-1$.
Consider a Gomory cut $c^\intercal x \leq \floor{\delta}$ that is valid for $P^{(i)}$, where $c^\intercal x \leq \delta$ is valid for $P^{(i-1)}$ and $c \in \mathbb{Z}_{\geq 0}^n\setminus\{\zero\}$. Since all unit vectors are feasible, and $c$ is nonzero, we have $\floor{\delta} \geq 1$. We consider the following two cases.\medskip

\noindent \emph{Case 1: $\floor{\delta} < \gamma/2^{i}$.} In this case, we have $2^i \floor{\delta} < \gamma$ and Lemma~\ref{lem:scaling} implies that $c^\intercal x \le \gamma$ is valid for $P$. Then, applying Lemma~\ref{lem:sampling} with $\eta := \epsbis/2$, and observing that $\alpha_i \ge (1+\epsbis/2)(1+\epsbis)$, we get that $c^\intercal x \le \alpha_i \floor{\delta}$ is valid for $Q$.\medskip

\noindent \emph{Case 2: $\floor{\delta} \ge \gamma/2^{i}$}. In this case, the induction hypothesis and $\delta \le \floor{\delta} + 1$ imply that $c^\intercal x \le \alpha_{i-1} (\floor{\delta}+1)$ is valid for $Q$. Now, observe that
\[
	\alpha_{i-1}(\floor{\delta}+1)
    \leq\bigl(\alpha_{i-1}+2^{i+1}/\gamma \cdot (1+\epsbis)\bigr)\floor{\delta}
    \leq\alpha_i\floor{\delta}\,
\]
where the first inequality follows from $\alpha_{i-1} \leq (1+3\epsbis/4)(1+\epsbis) \le 2(1+\epsbis)$. We again conclude that $c^\intercal x \le \alpha_i \floor{\delta}$ is valid for $Q$.\medskip

Taking $i = \numrounds$, and noticing that $\alpha_t = (1+\epsbis)^2$, we conclude that $Q$ is a $(1+\epsbis)^2$-approximation of $P^{(\numrounds)}$. We conclude the proof by setting $\epsbis:=\eps/3$, since $(1+\epsbis)^2 \leq 1+\eps$.
Thus, taking $r=\left\lceil(384\cdot2^t/\eps)\log(1+6/\eps)\right\rceil$ suffices.
\end{proof}
Theorem~\ref{thm:main-1} follows directly from Theorem~\ref{thm:local_to_CG} and Lemma~\ref{lem:local}, by setting $r$ to be the one given by Theorem~\ref{thm:local_to_CG}.

\section{Locally-near-exact relaxations from Gomory closures}\label{sec:locally-near-exact-from-Gomory}

In this section, we prove Theorem~\ref{thm:CG_to_local}. Once again, we may assume that $P$ contains the standard simplex (equivalently, that $A \in [0,1]^{m \times n}$) without loss of generality.

First we show the following lemma generalizing Proposition~3.2 of \cite{es03}. In our improved upper bound the number of variables $n$ gets replaced by $k$, which is the maximum \emph{mass} (that is, $\ell_1$-norm) of a feasible (fractional) point. Notice that $k \le n$ always holds. Let the \emph{depth} of an inequality $c^\intercal x \le \rhs$, which is valid for $P_I$, denote the minimal $\numrounds$ such that it is also valid for $P^{(\numrounds)}$.

\begin{lemma}\label{lem:generalized-depth-bound}
    Let $P\subseteq[0,1]^n$ be a polytope with $P\cap\set{0,1}^n\ne\emptyset$, and suppose that $\sum_{j=1}^n x_j\le k$ is valid for $P$, where $k\in\Z_{\ge0}$. Let $c^\intercal x\le\rhs$ be valid for $P_I$, where $c \in \Z^n_{\ge0}$ is nonzero. Then $c^\intercal x\le\rhs$ has depth at most
    \[
        D(k,\norm{c}_\infty):=k(k+1)+2k\lceil\log_2 \norm{c}_\infty \rceil.
    \]
\end{lemma}
Before proving Lemma~\ref{lem:generalized-depth-bound}, we establish the following lemma which bounds the rank of an integer-empty polytope (contained in the cube) in terms of $k$, and is similar to Lemma~2.2 from \cite{es03}.

\begin{lemma}\label{lem:integer-empty-cube-bound}
Let $Q\subseteq[0,1]^n$ be a polytope with $Q \cap \set{0,1}^n = \emptyset$. If $\sum_{j=1}^n x_j\le k$ is valid for $Q$ where $k\in\Z_{\ge0}$, then $Q^{(k)}=\emptyset$. 
\end{lemma}
\begin{proof}
We show this by induction on $k$. The case $k=0$ is immediate. Take any $i \in [n]$. Applying the induction hypothesis to the face $F_i := Q \cap \set{x \in \R^n : x_i=1}$, or more precisely the orthogonal projection of $F_i$ in $[0,1]^{n-1}$ obtained by deleting the $i$-th coordinate, we see that $F_i^{(k-1)} = \emptyset$.

Hence, unless $Q^{(k-1)}$ is already empty, we have $\max_{Q^{(k-1)}}x_i<1$ for every fixed $i \in [n]$. The next round of Gomory cuts forces all coordinates to zero, which implies $Q^{(k)} = \emptyset$ since $Q$ contains no integer point.
\end{proof}

We are now ready to prove Lemma \ref{lem:generalized-depth-bound}.
\begin{proof}[Proof of Lemma \ref{lem:generalized-depth-bound}]
    We may assume that \(\rhs=\max\set{c^\intercal z:z\in P\cap\set{0,1}^n}\). 
	We prove the lemma by induction on $k+\norm{c}_\infty$. The case $k=0$ is trivial. Now let $k\ge1$, and define
    \[
        s:=D(k,\norm{c}_\infty)-2k=k(k-1)+2k\lceil\log_2 \norm{c}_\infty \rceil, 
        \quad \text{ and } \quad
        \bar c:=\floor{c/2},
    \]
    where the floor is taken coordinate-wise.
    First, we show that $c^\intercal x\le\rhs+k$ is valid for $P^{(s)}$. If $\bar c = 0$, then $c \in \set{0,1}^n$ and $c^\intercal x \le \sum_{j=1}^n x_j \le k \le \rhs + k$ holds already for all $x \in P$. Now, assume that $\bar c\ne0$. Then,
    \[
        D(k,\norm{\bar c}_\infty)\le D(k,\norm{c}_\infty)-2k=s.
    \]
    By induction, $\bar c^\intercal x \le \bar{\rhs}$ is valid for $P^{(s)}$, where $\bar{\rhs} := \max\set{\bar c^\intercal z:z\in P\cap\set{0,1}^n}$.
    Since $c-2\bar c\in\set{0,1}^n$ (and in particular $2\bar c\le c$), for every $x\in P^{(s)}$ we have
    \begin{equation*}\label{eq:small-objective-gap}
        c^\intercal x
        \le2\bar c^\intercal x+\sum_{j=1}^n x_j
        \le2 \bar{\rhs} +k
        \le\rhs+k.
    \end{equation*}

    Next, we show that every $x\in P^{(s)}$ with some coordinate equal to one satisfies $c^\intercal x\le\rhs$. Fix $i\in[n]$ and consider $F_i:=P \cap \set{x \in \R^n : x_i=1}$. After deleting coordinate $i$, every point of face $F_i$ has total mass at most $k-1$. If $F_i$ has an integer point, then the inequality
    \[
        \sum_{j\ne i}c_jx_j\le\rhs-c_i
    \]
    is valid for its integer hull. By induction, the depth of this inequality is at most
    \[
        D(k-1,\norm{c}_\infty)=k(k-1)+2(k-1)\lceil\log_2 \norm{c}_\infty \rceil\le s,
    \]
    with no rounds needed in case $c_j = 0$ for all $j \neq i$. If $F_i$ has no integer point, it disappears after at most $k-1\le s$ rounds of Gomory cuts, by Lemma \ref{lem:integer-empty-cube-bound}. Since the Gomory closure commutes with taking faces, in either case
    \begin{equation}\label{eq:coordinate-faces}
        x\in P^{(s)},\quad x_i=1
        \quad\Longrightarrow\quad c^\intercal x\le\rhs.
    \end{equation}

    To finish, we prove by a backward induction on $\ell=k,k-1,\ldots,0$ that $c^\intercal x\le\rhs+\ell$ has depth at most $s+2(k-\ell)$.
    The case $\ell=k$ is the bound established above. Suppose that the claim holds for some $\ell\ge1$, and let $t:=s+2(k-\ell)$. We show that
    $c^\intercal x\le\rhs+\ell-1$ holds for $P^{(t+2)}$. Notice that, since $P$ has an integer point, $P^{(t)}$ is nonempty. Consider the face 
    \[
        F:=P^{(t)}\cap\set{x \in \R^n : c^\intercal x=\rhs+\ell}.
    \]
    If $F=\emptyset$, we have $\max_{x\in P^{(t)}}c^\intercal x<\rhs+\ell$.
    Since $c$ and $\rhs+\ell$ are integral, one round of Gomory cuts yields $c^\intercal x\le\rhs+\ell-1$.
    Now suppose $F\ne\emptyset$. By \eqref{eq:coordinate-faces} we may assume that $F$ contains no point with $x_i=1$, for any $i$.
    Then we have $\max_{x\in F}x_i<1$ for every $i \in [n]$, hence the corresponding Gomory cuts imply
    $F' \subseteq\set{\mathbf 0}$.
    Since $\rhs+\ell>0$, the origin does not belong to $F$,
    and therefore $F' =\emptyset$.
    By the face property of the Gomory closure,
    \[
        P^{(t+1)} \cap \set{x \in \R^n : c^\intercal x=\rhs+\ell}
        =F'
        =\emptyset.
    \]
    If $P^{(t+1)} \neq \emptyset$, we have $\max_{x\in P^{(t+1)}}c^\intercal x<\rhs+\ell$, and another round of Gomory cuts yields the validity of $c^\intercal x\le\rhs+\ell-1$ for $P^{(t+2)}$.
    If $P^{(t+1)} = \emptyset$, the conclusion is immediate.
    This completes the induction. For $\ell=0$, we conclude that $c^\intercal x\le\rhs$ is valid for $P^{(s+2k)}$. That is, $c^\intercal x\le\rhs$ has depth at most $s+2k = D(k,\norm{c}_\infty)$.
\end{proof}

We conclude this section with a proof of Theorem~\ref{thm:CG_to_local}. 

\begin{proof}[Proof of Theorem~\ref{thm:CG_to_local}]
	Consider any inequality $\sum_{j\in K}c_jx_j\le\rhs$, with $c\in\R^n_{> 0}$ and $\rhs\in\R_{\ge0}$, that is valid for $P_I$, where the support $K\subseteq N(T)$ for some $T = (i_1, \dots, i_r) \in [m]^r$ (with possible repetitions). Our goal is to prove that the inequality
    \[
        \sum_{j\in K}c_jx_j\le(1+\eps)\rhs
    \]
    is valid for $P^{(t)}$, where $t:=r(r+1)+2r\left\lceil\log_2(r\eps^{-1})\right\rceil$.
    In particular when $t=\Omega(r^2 + r\log(1/\eps))$, the above inequality is valid for $P^{(t)}$.

    We may assume that $A \in [0,1]^{m\times n}$ and $K \neq \emptyset$. Then for each $j\in K$, the unit vector $e_j$ is feasible, which implies that
    \begin{equation}\label{eq:singleton-objective-bound}
        0<\norm{c}_\infty\le\rhs.
    \end{equation}
    Also, for every $x\in P$, we have
    \[
        \sum_{j\in K}x_j
        \le\sum_{j\in K}\underbrace{\sum_{\ell=1}^{r} A_{i_\ell, j}}_{\ge 1}x_j
        \le\sum_{\ell=1}^{r}\sum_{j=1}^n A_{i_\ell, j}x_j
        \le r,
\]
    hence $\sum_{j\in K}x_j\le r$.
    For each $j\in K$, define
    \(
        c'_j:=\floor{\rho c_j}
    \)
    where 
    \(
        \rho:=\frac{r}{\eps\norm{c}_\infty}>0,
    \)
    and $c'_j := 0$ for $j \in [n] \setminus K$. Then
    \(
        1\le\norm{c'}_\infty\le \rho \norm{c}_\infty =r/\eps
    \).
    Define
    \(
        \rhs':=\max\set{\sum_{j\in K}c'_jz_j:
            z\in P\cap\set{0,1}^n}.
    \)
    Since $c'\le\rho c$, validity of the original inequality gives $\rhs'\le\rho\rhs$. Consider the inequality
    \begin{equation}\label{eq:dominating-ineq}
        \sum_{j\in K}c'_jx_j\le\rhs'.
    \end{equation}

    First, we show that~\eqref{eq:dominating-ineq}, together with $x\in P$, implies that $\sum_{j\in K}c_jx_j\le(1+\eps)\rhs$. Indeed, $\rho c_j<c'_j+1$ for each $j\in K$. Therefore, for every $x\in P$ satisfying~\eqref{eq:dominating-ineq},
    \begin{align*}
        \sum_{j\in K}c_jx_j
        &\le\frac1\rho\rb{\sum_{j\in K}c'_jx_j+\sum_{j\in K}x_j}
        \le\frac1\rho\rb{\rhs'+r}\\
        &\le\rhs+\frac{r}{\rho}
        =\rhs+\eps\norm{c}_\infty
        \le(1+\eps)\rhs,
    \end{align*}
    where the last inequality follows from~\eqref{eq:singleton-objective-bound}.

    Second, we show that~\eqref{eq:dominating-ineq} has depth at most $t$. Let $Q:= \pi_K(P)$ be the orthogonal projection of $P$ onto the coordinates $x_j$ with $j \in K$. Since $P$ is packing, every integer point of $Q$ extends to an integer point of $P$, for instance setting $x_j = 0$ for $j \in [n] \setminus K$. Thus~\eqref{eq:dominating-ineq} is valid for the integer hull of $Q$. Moreover, $Q\subseteq[0,1]^K$ and $\sum_{j\in K}x_j\le r$ is valid for $Q$. By Lemma~\ref{lem:generalized-depth-bound}, the inequality is valid for $Q^{(t)}$, since
    \[
        D(r,\norm{c'}_\infty)
        =r(r+1)+2r\lceil\log_2\norm{c'}_\infty\rceil
        \le r(r+1)+2r\left\lceil\log_2(r\eps^{-1})\right\rceil=t.
    \]
    Since $\pi_K\bigl(P^{(t)}\bigr) \subseteq Q^{(t)}$, the inequality $\sum_{j\in K}c_jx_j\le(1+\eps)\rhs$ is also valid for $P^{(t)}$. 
\end{proof}

\section{Polyhedral approximations of Gomory closures from communication protocols}\label{sec:depth-independent}

Recall that Mastrolilli~\cite{mas20} provided a convex, but non-polyhedral relaxation that approximates any fixed-rank Gomory closure of a packing polytope. Before we proceed with our communication protocol, we explain a key idea in
his work that will also play a central role here.
Mastrolilli focuses on Gomory cuts with small right-hand side and combines this with an induction on $\numrounds$. This is formalized in the next lemma. A similar statement appears in~\cite{mas20}. We include a proof for completeness.

\begin{lemma} \label{lem:induction}
Let $P$ be a packing polytope. Let $\epsbis > 0$ and let $\numrounds \in \Z_{\ge 1}$. If $Q \subseteq \R^n_{\ge 0}$ is a convex relaxation of $P_I$ such that 
%
\begin{enumerate}[(i)]
\item $Q$ is a $(1+\epsbis)^{\numrounds-1}$-approximation of $P^{(\numrounds-1)}$,
\item for every inequality $c^\intercal x \le \floor{\delta}$ that is valid for $P^{(\numrounds)}$ with $c \in \Z^n_{\ge 0}$ and $\floor{\delta} < 1/\epsbis$, the inequality $c^\intercal x \le (1+\epsbis)^\numrounds \floor{\delta}$ is valid for $Q$,
\end{enumerate}
then $Q$ is a $(1+\epsbis)^{\numrounds}$-approximation of $P^{(\numrounds)}$.
\end{lemma}

\begin{proof}
Consider any valid inequality $c^\intercal x \le \floor{\delta}$ for $P^{(\numrounds)}$ where $c \in \Z^n_{\ge 0}$ and $c^\intercal x \le \delta$ is valid for $P^{(\numrounds-1)}$. Our goal is to show that $c^\intercal x \le (1+\epsbis)^{\numrounds} \floor{\delta}$ is valid for $Q$. 
	If $\floor{\delta} < 1/\epsbis$, this holds by (ii). Hence, we may assume that $\floor{\delta} \ge 1/\epsbis$. Since $c^\intercal x \le \delta$ is valid for $P^{(\numrounds-1)}$ and $Q$ is a $(1+\epsbis)^{\numrounds-1}$-approximation of $P^{(\numrounds-1)}$, the inequality $c^\intercal x \le (1+\epsbis)^{\numrounds-1} \delta$ is valid for $Q$. 
Using that $\delta \le \floor{\delta} + 1 \le (1+\epsbis) \floor{\delta}$, we conclude that $c^\intercal x \le (1+\epsbis)^{\numrounds} \floor{\delta}$ is valid for $Q$.
\end{proof}


As observed by Mastrolilli~\cite{mas20}, the $(\lfloor 1/\epsbis \rfloor+1)$-th level of the Lasserre hierarchy yields such a relaxation $R$. We remark that this also holds for the earlier positive semidefinite Lov\'asz-Schrijver hierarchy, which follows from \cite[Lemma 1.5]{LovaszSchrijver90}. In both cases, the final relaxation $Q = P \cap R$ can be described by a polynomial-size 
\emph{positive semidefinite} extended formulation. This leaves the open question whether a $(1+\eps)$-approximation of $P^{(\numrounds)}$ can be achieved with a polynomial-size \emph{linear} extended formulation, for all fixed $\eps > 0$ and $t \in \Z_{\ge 1}$. This question is motivated by the fact that linear programming is conceptually much simpler than semidefinite programming and admits highly scalable solvers in practice.

As we show below, Mastrolilli's approach can be scaled down from positive semi\-definite to linear extended formulations, see Theorem~\ref{thm:main-2}. However, the size upper bound jumps from polynomial to quasi-polynomial. 

Towards Theorem~\ref{thm:main-2}, the main technical challenge is to prove the following lemma about capturing all valid inequalities for the integer hull with small right-hand sides. 

\begin{lemma} \label{lem:small-rhs}
Let $P = \{x \in \R^n_{\ge 0} : Ax \le b\}$ be a polytope contained in $[0,1]^n$, where $A \in \R_{\ge 0}^{m \times n}$ and $b \in \R_{\ge 0}^m$. Let $\rhs_{\max} \in \Z_{\ge 1}$. There exists a polyhedral relaxation $R$ of the integer hull of $P$ such that $R$ has a size-$n^{2
\rhs_{\max} \log n}$ extended formulation and every inequality $c^\intercal x \le \rhs$ with integer coefficients and $\rhs \le \rhs_{\max}$ that is valid for the integer hull of $P$ is also valid for $R$.
\end{lemma}

\begin{proof}
We give a deterministic protocol with complexity $O(\rhs_{\max} \log^2 n)$ which solves the following communication problem:
There are two players Alice and Bob. Alice is given an inequality $c^\intercal x \le \rhs$ that is valid for the integer hull of $P$, where $c \in \Z^n_{\ge 0}$, $\rhs \in \{1,\ldots,\rhs_{\max}\}$. We write $K:=\supp(c)$. Bob is given a feasible integer solution $z \in P \cap \Z^n = P \cap \{0,1\}^n$. We write $B:=\supp(z)$. Their task is to compute the \emph{slack} $\rhs - c^\intercal z \in \{0,\ldots,\rhs\}$ by exchanging as few bits as possible. In this context, they do so deterministically. The number of bits exchanged in the worst case is the \emph{complexity} of the protocol. Each player knows $P$ and their own input, but does not know the input of the other player. The players are otherwise not assume to be computationally bounded.

By known results in~\cite{yan88}, which connect communication protocols and extended formulations, our protocol gives an extended formulation for $Q$ of size $2^{O(\rhs_{\max} \log^2 n)} = n^{O(\rhs_{\max} \log n)}$, i.e.,
singly exponential in the \emph{complexity} of the protocol.
	. In fact, our protocol generalizes Yannakakis' protocol~\cite{yan88} for the clique versus stable set problem, to any packing problem and any right-hand side.

For any set $S \subseteq [n]$ with $|S| \le \rhs$, we define the $\forcingset(S)$ to be the set of all indices $j' \in [n]$ such that $x_{j’} = 0$ in each feasible integer solution $x$ such that $x_j = 1$ for all $j \in S$. 
Define $\forcing(S)$ to be the size of its forcing set. We call a set $S \subseteq K$ \emph{tight} if $\sum_{j \in S} c_j = \rhs$ and $\mathbf 1_S \in P$. Notice that tight sets have at most $\rhs$ elements, since $c_j \ge 1$ for all $j\in K$.

In the protocol, Alice and Bob try to apply a number of simplifying rules in the given order until one of them succeeds. Then
they start again at the first rule.
In this process, they may identify elements $j\in K\cap B$, which they then contract (set to $1$).
This means they both remove it from the instance, Alice reduces $\rhs$ by $c_j$ and they compute the slack in
the remainder instance, which they then output for the overall instance as well.
\medskip

\noindent \textbf{Rule 1.} If $n = 0$ or $\rhs = 0$ then the protocol stops and Alice outputs $\rhs$. 
\medskip

\noindent \textbf{Rule 2.} Bob looks for a set $S \subseteq B$ of size at most $\rhs$ with $\forcing(S) \ge n/2$. 
If he finds such a set, he informs Alice of this and sends all elements in $S$ one after the other.
After each element $j$, Alice tells Bob whether $j\in K$.
Then, they both contract (set to $1$) the variables $x_j$ for $j \in K \cap S$ and delete (set to $0$) the variables $x_j$ for $j \in \forcingset(S)$, and recurse on a smaller problem with at most $n/2$ variables and right-hand side at most $\rhs$. This is safe since $\forcingset(S) \cap B = \varnothing$.
If Bob fails to find such a set $S \subseteq K$ with $\forcing(S) \ge n/2$, he informs Alice and the protocol continues to the next rule.
\medskip

\noindent \textbf{Rule 3.} Alice looks for a tight set $S \subseteq K$ with $\forcing(S) < n/2$. Recall that such a set satisfies $1 \le |S| \le \rhs$. If she can find such a set, we proceed similar to before: she sends to Bob all the elements $j\in S$ and Bob tells Alice whether $j \in B \cap S$ or not. Afterwards, both contract the variables $x_j$ for $j \in B \cap S$ and delete the variables $x_j$ for $j \notin \forcingset(S)$, and recurse on a smaller problem with at most $n/2$ variables and right-hand side at most $\rhs$. This is again safe since $K\setminus S \subseteq \forcingset(S)$. If Alice fails to find a tight set $S \subseteq B$ with $\forcing(S) < n/2$, she informs Bob of this and they continue the protocol.

Suppose that Rules 1--3 have been applied until exhaustion. Then the following properties hold:
\begin{enumerate}[(P1)]
\item every $S \subseteq B$ with $|S| \le \rhs$ has $\forcing(S) < n/2$;
\item every tight set $S \subseteq K$ has $\forcing(S) \ge n/2$.
\end{enumerate}

\medskip

\noindent \textbf{Rule 4.} If none of the players is able to find a set $S$ contained in their respective set which allows them to make progress, they add to the inequalities defining $P$ all the inequalities of the form $\sum_{j \in S} x_j \le |S| - 1$ where $S \subseteq [n]$ has size at most $\rhs$ and $\forcing(S) \ge n/2$. They can do so without communicating. Let $\tilde{P} \subseteq P$ denote the resulting polytope. By (P1), $z$ is an integer point of $\tilde{P}$. By (P2), $c^\intercal x \le \rhs - 1$ is valid for the integer hull of $\tilde{P}$: indeed, otherwise there would exist an integer solution $y \in \tilde{P} \cap \{0,1\}^n$ with $c^\intercal y > \rhs - 1$. Since $\tilde{P} \subseteq P$, we have $y \in P \cap \{0,1\}^n$, implying that $c^\intercal y \le \rhs$ and therefore $c^\intercal y = \rhs$. Let $S := K \cap \supp(y)$. By (P2), we have $\forcing(S) \ge n/2$. Hence the constraint $\sum_{j \in S} x_j \le |S| - 1$ is one of the constraints defining $\tilde{P}$, however it is violated by $y$ since $y_j = 1$ for all $j \in S$, this contradicts our assumption that $y \in \tilde{P}$. The players can recurse on $\tilde{P}$ with inputs $c^\intercal x \le \rhs - 1$ (for Alice) and $z$ (for Bob). To the computed slack, they then add $1$ and output it.

Let $f(n,\rhs)$ denote the number of bits exchanged by the protocol above. Then $f(n,\rhs) = 0$ when $n = 0$ or $\rhs = 0$, and $f(1,\rhs) = 1$. Otherwise, we have
\begin{align*}
f(n,\rhs) \le \max \{ &1 + \rhs (\ceil{\log n}+1) + f(\floor{n/2},\rhs),\\ 
       &1 + 1 + f(n,\rhs-1)\}.
\end{align*} 

We infer from this that $f(n,\rhs) \le 2 \rhs \log^2 n$ for $n \ge 2$ and $\rhs \ge 1$.
\end{proof}

We are now ready to prove Theorem~\ref{thm:main-2}, which we restate here for the convenience.

\ef*

\begin{proof}
	Let $\epsbis := \ln(2) \cdot \eps/t$ and let $Q := P \cap R$, where $R$ is the relaxation from Lemma~\ref{lem:small-rhs} with $\delta_{\max} := \lfloor 1/\epsbis \rfloor$. We prove that $Q$ is a $(1+\epsbis)^i$-approximation of $P^{(i)}$, by induction on $i \in \{0,\ldots,\numrounds\}$. Since $(1+\epsbis)^t = (1+\ln(2) \cdot \eps/t)^t \le \mathrm{e}^{\ln(2)\eps} = 2^\eps \le 1 + \eps$, this implies in particular that $Q$ is a $(1+\eps)$-approximation of $P^{(\numrounds)}$.
	
For $i = 0$, we have $Q \subseteq P$ hence $Q$ is a $(1+\epsbis)^0$-approximation of $P^{(0)} = P$.

Now assume that $i \ge 1$, and invoke Lemma~\ref{lem:induction}. By induction, we know that $Q$ is a $(1+\eps')^{i-1}$-approximation of $P^{(i-1)}$. Hence, Condition~(i) of Lemma~\ref{lem:induction} is satisfied. Moreover, $Q \subseteq R$. By Lemma~\ref{lem:small-rhs}, this in particular implies that Condition~(ii) of Lemma~\ref{lem:induction} is satisfied. We conclude that $Q$ is a $(1+\eps)^{i}$-approximation of $P^{(i)}$.
\end{proof}

We conclude this section by discussing the size-optimality of our (linear) extended formulation in Theorem~\ref{thm:main-2}.
It follows from 
Göös \emph{et al.}~\cite{gpw15} that the complexity of every deterministic protocol that computes the slack matrix is at least $\Omega(\rhs_{\max} \cdot \log^2 n)$. Moreover, even \emph{approximating} the valid inequalities with $\rhs_{\max} = 1$ requires size $n^{\Omega(\log n)}$ in any extended formulation, see~\cite{gjpw18}. This shows that the $n^{O(\log n)}$ term in 
Theorem~\ref{thm:main-2} is necessary. 
Furthermore, it is not possible to remove the dependence on $m$. Indeed, suppose that $P$ is the stable set polytope of some graph. Then one cannot hope to find a quasi-polynomial size extended formulation providing a $(1+\eps)$-approximation of $P' = P$. This follows by combining results from Bazzi \emph{et al.} \cite{bfps15} and Kothari \emph{et al.} \cite{kmr17}. Indeed, there are $n$-vertex graphs $G$ for which every extended formulation providing (say) a $2$-approximation of the stable set polytope of $G$ has size at least $2^{\Omega(n^\rhs)}$, for some constant $\rhs > 0$.

\section*{AI Statement}

The first version of this paper was uploaded to arXiv on July 10, 2026. For this, the authors did not use AI at all, except for creating a drawing and for literature research. 

The second (current) version has benefited from improvements that were suggested by OpenAI's GPT-6 Astra\footnote{We refer to this model as ``AI'' below, since it is the only AI model used for this paper.}. We list them below. 
\begin{itemize}
    \item The exponent in the bound of Theorem~\ref{thm:main-1} was improved from $(t/ε)^{O(t)}$ to $O\left(2^t/ε \log (1/ε)\right)$. The proof suggested by AI has become much  shorter.  For the particularly relevant case of the first Gomory closure, the bounds of the first and current versions are equal. 
    \item  The generalization of \cite[Proposition 3.2]{es03} stated in Lemma~\ref{lem:generalized-depth-bound} was
      explicitly \emph{prompted} by the authors. The proof of Theorem~\ref{thm:CG_to_local} is entirely the work of humans.      
    \item (Appendix~\ref{sec:quasi-poly-construction}) The authors prompted AI  to use the approach in~\cite{AprileFaenza20}  to find an algorithm that constructs the relaxation in Theorem~\ref{thm:main-2} in quasi-polynomial time.
      \item (Appendix~\ref{sec:hardness}) Theorem~\ref{thm:hardness}, which is a specialization of Theorem~26 of Cornuéjols, Lee and Li~\cite{cornuejols2020rational}, was found by AI. 
\end{itemize}
The authors have verified and rewritten all AI-generated proofs and take full responsibility for all content of this paper.

\section*{Acknowledgements}
Friedrich Eisenbrand was supported by project 10000183 of the Swiss
National Science Foundation (SNSF).
Samuel Fiorini was supported in parts by \emph{Fonds de la Recherche Scientifique - FNRS} through research project BD-DELTA-3 (PDR 40028812), and by the King Baudouin Foundation through Thelam fund BD-DELTA-2 (2025-F2150080-0026829).
Lars Rohwedder was supported by ERC grant PARAMLP (project no. 101220563) funded by the European Union.

\phantomsection
\addcontentsline{toc}{section}{References}
\printbibliography

\appendix

\section{Quasi-polynomial time construction}\label{sec:quasi-poly-construction}
In the setting of Theorem~\ref{thm:main-2}, assume that $n\ge2$ and $A$
is rational, and let $L$ denote its total binary encoding length.
\begin{theorem}
\label{thm:constructive-main-2}
The relaxation in Theorem~\ref{thm:main-2} can be constructed in time
\[
  \operatorname{poly}(L,t/\eps)\,n^{O(t/\eps\cdot \log n)}.
\]
\end{theorem}

\begin{proof}
Let $\delta_{\max}=\lfloor 2t/\eps \rfloor$.
By the proof of Theorem~\ref{thm:main-2}, it suffices to construct a
relaxation $R\subseteq[0,1]^n$ of the integer hull of $P$ as in
Lemma~\ref{lem:small-rhs}, i.e., it satisfies every valid inequality
$c^{\intercal}x\le\delta$ for the integer hull of $P$ with
$c\in\Z_{\ge0}^n$ and $\delta\in\{0,\ldots,\delta_{\max}\}$.

We construct such a relaxation $R$ in two steps.
First, we explicitly build an auxiliary tree of possible message histories
from the communication protocol in Lemma~\ref{lem:small-rhs}, assigning
to each leaf the face of $[0,1]^n$ specified by the coordinates recorded
as zero or one along its path.
Second, following the proof of Aprile and
Faenza~\cite[Theorem~5]{AprileFaenza20}, we assemble these faces
from the leaves to the root, taking intersections at Alice nodes
and convex hulls at Bob nodes. The resulting polytope at the root is $R$.

First, we construct recursively the auxiliary tree.
At each node $v$, keep the current packing system $P_v$ on coordinates
$U_v$, the sets $C_v,Z_v$ of original coordinates recorded as one and
zero, and a counter $r_v$ for applications of Rule~4.
At the root these are $P,[n],\varnothing,\varnothing,0$.
Before each round, substitute $x_j=0$ and remove $j$ whenever
$e_j\notin P_v$, adding these coordinates to $Z_v$.
If $U_v=\varnothing$ or $r_v=\delta_{\max}$, make the node a leaf.

We determine the branches by direct feasibility checks.
For each feasible $S\subseteq U_v$ with $|S|\le\delta_{\max}$, compute
\[
  \forcingset(S)=\{j\in U_v:\mathbf1_{S\cup\{j\}}\notin P_v\}.
\]
Call $S$ high-forcing if $|\forcingset(S)|\ge |U_v|/2$, and low-forcing
otherwise. At a nonterminal node, create the following branches;
the data remain unchanged unless an update is specified.
\begin{enumerate}
\item
Create an Alice node with a \emph{stop} child, which is a leaf,
and a \emph{continue} child, which is the Bob node in item~2.

\item
At this Bob node, create one branch for every high-forcing $S$.
Along this branch, substitute $x_S=1$ and $x_{\forcingset(S)}=0$ in $P_v$,
remove these coordinates from $U_v$, add $S$ to $C_v$ and
$\forcingset(S)$ to $Z_v$, and start a new round.
Also create a \emph{fail} branch leading to the Alice node in item~3.

\item
At this Alice node, create one branch for every low-forcing $S$,
leading to a Bob node with one branch for every $T\subseteq S$.
Along the branch indexed by $T$, substitute
\[
  x_T=1,\qquad x_{U_v\setminus(\forcingset(S)\cup T)}=0
\]
in $P_v$ and retain only the coordinates $\forcingset(S)$.
Add $T$ to $C_v$ and $S\setminus T$ to $Z_v$, then start a new round.
The other removed coordinates are not added to $Z_v$ and remain
free in the eventual leaf face.
Also give this Alice node a \emph{fail} branch: add
\[
  x(S)\le |S|-1
  \qquad\text{for every high-forcing }S,
\]
increase $r_v$ by one, and start a new round.
\end{enumerate}

We obtain the whole tree by repeating these rules at every new round.
Both failure branches and the stop branch are always included;
no condition involving $c$, $z$, or tightness is tested.
Every substitution and feasibility check uses the explicitly listed
inequalities of $P_v$.
Zero remains feasible: all sets substituted by one are feasible,
and preprocessing ensures that the empty set is never high-forcing.
The updates in items~2 and~3 are those of Rules~2--4 in
Lemma~\ref{lem:small-rhs}, with the uniform bound $\delta_{\max}$
and with all of Bob's set $S$ substituted by one.
For every leaf $v$, define
\[
  T_v=\{x\in[0,1]^n:x_j=1\, \forall\, j\in C_v,\; x_j=0\, \forall\, j\in Z_v\}.
\]

Second, we assemble an extended formulation from the leaf faces.
Following the proof of
Aprile and Faenza~\cite[Theorem~5]{AprileFaenza20}, take intersections
at Alice nodes and convex hulls at Bob nodes.
Intersections use the child formulations with shared original
variables; convex hulls use the formulation recalled in
\cite[Theorem~4]{AprileFaenza20}.
All resulting polytopes are nonempty: choosing failure or
$T=\varnothing$ at Bob nodes preserves $\mathbf1_{C_v}$ throughout
the subtree rooted at $v$.
Denote the polytope at the root by $R$.

We first verify that $R$ contains the integer hull of $P$.
Fix an integer feasible point $z$.
As in Lemma~\ref{lem:small-rhs}, let Bob choose a high-forcing
$S\subseteq\supp(z)$ when one exists, fail otherwise, and answer
$T=S\cap\supp(z)$ at a subset node.
Against every Alice choice, the remaining point stays feasible
and all recorded coordinates agree with $z$.
In particular, Bob's failure ensures that his point satisfies every
inequality added after double failure.
Thus every reached leaf face contains $z$.
At an Alice node this holds for every child, and at a Bob node
for a chosen child; induction from the leaves gives $z\in R$.

We next verify the target inequalities using Alice's choices.
Fix a valid $c^{\intercal}x\le\delta$ with
$c\in\Z_{\ge0}^n$ and $\delta\le\delta_{\max}$.
Alice follows the proof of Lemma~\ref{lem:small-rhs}:
she stops when the current right-hand side $\delta_v$ is zero,
and otherwise chooses a low-forcing tight set, or fails if none exists.
A tight set can be chosen inside $\supp(c)$ and then has size
at most $\delta_v$.
Feasible substitutions subtract the coefficients of the coordinates
set to one.
For a tight $S$, every removed but unrecorded coordinate $j$ has
$c_j=0$, since $\mathbf1_{S\cup\{j\}}$ is feasible.
If Alice fails, every tight integer point contains a high-forcing
tight set and is excluded by the added inequalities; integrality
therefore allows decreasing $\delta_v$ by one.
Consequently the current inequality remains valid and
\[
  \delta=c(C_v)+\delta_v+r_v,\qquad \delta_v\ge0,
  \qquad c(C_v):=\sum_{j\in C_v}c_j.
\]

We check the inequality on each leaf reached under these choices.
At such a leaf, either no coordinates remain or $\delta_v=0$;
the latter also follows from the displayed identity if
$r_v=\delta_{\max}$.
When $\delta_v=0$, preprocessing ensures that all remaining unit
vectors are feasible, so their coefficients are zero.
Previously removed free coordinates also have zero coefficients.
Hence $c^{\intercal}x=c(C_v)\le\delta$ throughout the leaf face.
An Alice intersection is contained in her chosen child, while a
Bob convex hull preserves an inequality satisfied by every child.
Induction therefore proves $c^{\intercal}x\le\delta$ on $R$.

Finally, we bound the construction time.
Each successful round at least halves the number of remaining
coordinates, and each double failure increases $r_v$.
Thus every path has $O(\log n)$ successful rounds and at most
$\delta_{\max}$ double failures.
Set choices take $O(\delta_{\max}\log n)$ bits per successful
round, so the tree has
\[
  N=2^{O(\delta_{\max}\log^2 n)}
   =n^{O(\delta_{\max}\log n)}
\]
nodes.
At each node we enumerate $n^{O(\delta_{\max})}$ sets in a system
with at most $m+2n+\delta_{\max}n^{O(\delta_{\max})}$ inequalities;
substitutions preserve polynomial coefficient bit lengths.
Each leaf face has $O(n)$ constraints, and the assembly produces
an extended formulation of size $O(nN)$ in polynomial time in
its description length.
The total time is therefore
$\operatorname{poly}(L,\delta_{\max})
 n^{O(\delta_{\max}\log n)}$.
Taking $Q=P\cap R$ gives the required relaxation, and substituting
$\delta_{\max}=O(t/\eps)$ proves the theorem.
\end{proof}

\section{Hypergraph matching}\label{sec:hypergraph-matching}
Given a hypergraph $G = (V, E)$, the maximum matching problem on $G$ is to find a maximum cardinality subset $E^\prime \subseteq E$ of hyperedges so that each vertex $v\in V$ is incident to at most one hyperedge $e \in E^\prime$.
For the unweighted maximum matching problem in a hypergraph, we have the following standard LP relaxation, where $\delta(v):=\{e\in E:v\in e\}$:
    \begin{align*}
      \max & \quad \sum_{e\in E} x_e \\
      \text{s.t.}
      & \quad \sum_{e\in\delta(v)} x_e \leq 1, \forall \, v\in V,\\
      & \quad x_e \geq 0, \forall \, e\in E.
    \end{align*}
A hypergraph $G = (V, E)$ is called $k$-uniform if $|e|=k$ for each $e \in E$. In this section, we study the (unweighted) maximum matching problem for $k$-uniform hypergraphs and show the following.

\begin{theorem}\label{thm:hyper}
    For the maximum matching problem in $k$-uniform hypergraphs $G=(V,E)$, it suffices to apply $O(\log k)$ iterations of the Gomory-closure operation to the standard LP relaxation to reduce the integrality gap to $(k+1)/2$.
\end{theorem}

\begin{proof}
    It was shown in \cite{cl12}, that the following LP has integrality gap of at most $(k+1)/2$:
    \begin{align*}
      \max & \quad \sum_{e\in E} x_e \\
      \text{s.t.}
      & \quad \sum_{e\in\delta(v)} x_e \leq 1, \forall \, v\in V, \\
      & \quad \sum_{e\in K} x_e \leq 1, \forall\;\text{intersecting family}\; K\subseteq E, \\
      & \quad x_e \geq 0, \forall \, e\in E.
    \end{align*}
    Let $K \subseteq E$ be an intersecting family, that is, a family of hyperedges with pairwise intersection.
    We claim that the depth of the clique inequality $x(K):=\sum_{e\in K} x_e\leq 1$ is $2+2\log r$, provided that $K$ has a hitting set of size $r$.
	By relabelling, assume that the vertices $\{1,2,\dotsc,r\} \subseteq V$ form a hitting set of $K$. Let $K_i$, $i\in\{1,2,\dotsc,r\}$, denote the edges in $K$ that contain vertex $i$.
	Then $K_1\cup \cdots\cup K_r = K$. Note that $x(K_i)\leq 1$ is valid for $P$ for each $i = 1,\dots, r$. Set $k_i := |K_i|$.
    Consider $K_1$ and $K_2$ (assume that both are non-empty). Since $K$ is a clique, we have that $x_i + x_j\leq 1$ is valid for $P$, for each $i \in K_1, j \in K_2$. 
    Hence for any fixed $i \in K_1$, we have that 
    \[x_i + \frac{1}{k_2} x(K_2) \leq 1\] 
    is valid for $P$. Also, 
    \[\frac{k_2-1}{k_2} x(K_2) \leq \frac{k_2-1}{k_2}\] 
    is valid for $P$. Hence by adding both inequalities we get that 
    \[x_i + x(K_2) \leq 1 + \frac{k_2-1}{k_2}\] 
    is valid for $P$. Thus, $x_i + x(K_2) \leq 1$ is valid for the first closure $P^\prime$ for every $i \in K_1$. 
    Similarly, $x(K_1) + x_j \leq 1$ is valid for $P^\prime$, for every $j \in K_2$. 
    Now we can average these inequalities to see that 
    \begin{equation*}
        \frac{1}{k_1} x(K_1) + x(K_2) \leq 1\quad \text{and} \quad x(K_1) + \frac{1}{k_2} x(K_2) \leq 1
    \end{equation*} 
    are valid for $P^\prime$. Without loss of generality, we may assume that $k_1\leq k_2$, hence by summing the two inequalities above and scaling by $k_2/(k_2+1)$ we have
    \[
	    \frac{k_1+1}{k_1}\cdot \frac{k_2}{k_2+1} x(K_1) + x(K_2) \le \frac{k_2}{k_2+1} x(K_1) + \left(\frac{1}{k_1} x(K_1) + x(K_2)\right) \leq \frac{k_2}{k_2+1} + 1 < 2 \, ,
    \]
    which implies that $x(K_1) + x(K_2) \leq 1$ is valid for the second closure $P^{\prime\prime}$.
    This means that after two rounds of Gomory-closure operation, we can ``merge'' every two consecutive cliques in $K_1, \cdots, K_r$ into a single clique. 
    By repeatedly merging, we get a bound of $2 \ceil{\log r} \leq 2+2\log r$ on the rank of $x(K) \leq 1$.

    By taking the $k$ vertices of one hyperedge in $K$ we can see that $K$ always has a hitting set of size $k$.
	Thus, we can take $r=k$. The claim above implies that $x(K)\leq 1$ can be captured in $O(\log k)$ rounds of Gomory-closure operation. This completes the proof.
      \end{proof}

\section{Optimizing over the Gomory closure of a packing polytope is \NP-hard}\label{sec:hardness}
\begin{theorem}[{\cite[Theorem~26]{cornuejols2020rational}}]\label{thm:np-hard-general}
	Let $A\in \R^{m\times n}$ and $b\in \R^m$. Consider $P = \{x \in [0, 1]^n \mid Ax \le b\}$. It is \NP-complete to decide whether the Gomory closure of $P$ is empty.
\end{theorem}

We will show that this implies that the optimization version is \NP-hard already for packing polytopes.
Note that testing emptiness is not of interest here, since the polytope and its Gomory closure
always contain $\mathbf 0$.

\begin{theorem}\label{thm:hardness}
	Let $A\in \R_{\ge 0}^{m\times n}$ and $b\in \R_{\ge 0}^m$. Consider the packing polytope $P = \{x \in [0, 1]^n \mid Ax \le b\}$ and its Gomory closure $P'$.
	It is \NP-hard to decide whether $\max\{\sum_{i=1}^n x_i : x\in P'\} \ge n/2$.
\end{theorem}
\begin{proof}
	Let $A\in \R^{m\times n}$, $b\in \R^m$, and $P = \{x \in [0, 1]^n \mid Ax \le b\}$ as in Theorem~\ref{thm:np-hard-general}, that is, $P$ is not necessarily a packing polytope.
	We will reduce the problem of testing emptiness of $P'$ to optimization in the packing setting. 
	We may assume without loss of generality that $P$ is not empty.
	Let $A^+$ be the $m\times n$ matrix with $A^+_{ij} = \max\{A_{ij}, 0\}$. Accordingly, $(-A)^+$ is the matrix with $(-A)^+_{ij} = \max\{-A_{ij},0\}$.

	Consider the packing polytope
	\begin{equation*}
		Q = \{(y, z)\in [0,1]^n \times [0,1]^n \mid A^+ y + (-A)^+ z \le b + (-A)^+ \mathbf 1, y_i + z_i \le 1 \forall i\in [n]\} . 
	\end{equation*}
	Polytopes $P$ and $Q$ are closely related: We have that
	\begin{equation}\label{eq:np-equiv}
		x \in P \Leftrightarrow (x, \mathbf 1 - x) \in Q .
	\end{equation}
	This equivalence is straight-forward: For the first direction, let $x\in P$. Then $A^+ x + (-A)^+ (\mathbf 1 - x) = A x + (-A)^+ \mathbf 1 \le b + (-A)^+ \mathbf 1$ and $x_i + (1 - x_i) \le 1$ for all $i\in [n]$. Thus, $(x, \mathbf 1 - x) \in Q$.
	Conversely, assume that $(x, \mathbf 1 - x) \in Q$. Then $Ax = A^+ x + (-A)^+ (\mathbf 1 - x) - (-A)^+ \mathbf 1 \le b + (-A)^+ \mathbf 1 - (-A)^+ \mathbf 1$. Thus, $x\in P$.

	In order to finish the proof, we will further show that $\max\{\sum_{i=1}^n y_i + z_i \mid (y,z)\in Q'\} \ge n$ if and only if $P'\neq \emptyset$, which indeed suffices to prove the theorem statement.

	Consider first $P' \neq \emptyset$ and let $x\in P'$. We will show that $(x, \mathbf 1 - x)\in Q'$
	and therefore $\max\{\sum_{i=1}^n y_i + z_i \mid (y,z)\in Q'\} \ge n$.
	We need to verify that $(x, \mathbf 1 - x)$ satisfies every Gomory Cut
	given by some $(c,d)\in \Z^n\times \Z^n$ and $\delta\in\Z$. We assume without loss of generality that
	$\delta = \lfloor \max\{c^\intercal y + d^\intercal z \mid (y, z) \in Q\} \rfloor$.
	Notice that since $x\in P'$, we have
	\begin{equation*}
		(c-d)^\intercal x \le \lfloor \max\{(c-d)^\intercal y \mid y \in P\} \rfloor .
	\end{equation*}
	Thus,
	\begin{align*}
		c^\intercal x + d^\intercal (\mathbf 1 - x)
		&= (c-d)^\intercal x + d^\intercal \mathbf 1 \\
		&\le \lfloor \max\{(c-d)^\intercal y \mid y \in P\} \rfloor + d^\intercal \mathbf 1 \\
		&= \lfloor \max\{(c-d)^\intercal y + d^\intercal \mathbf 1 \mid y \in P\} \rfloor \\
		&= \lfloor \max\{c^\intercal y + d^\intercal (\mathbf 1 - y) \mid (y, \mathbf 1-y) \in Q\} \rfloor \\
		&\le \lfloor \max\{c^\intercal y + d^\intercal z \mid (y, z) \in Q\} \rfloor = \delta .
	\end{align*}
	The second-to-last equation follows from~\eqref{eq:np-equiv}.

	Now suppose that there exists some $(y, z)\in Q'$ with $\sum_{i=1}^n y_i + z_i = n$.
	From $y_i + z_i\le 1$ for all $i$ it follows that $z = \mathbf 1 - y$.
	We will show that $y \in P'$, that is, $y$ satisfies all Gomory cuts given by $c\in \Z^n$, $\delta\in\Z$.	Here, we assume without loss of generality that $\delta = \lfloor \max\{c^\intercal x \mid x\in P\} \rfloor$.
	Let $M\in \Z$ be a larger integer that we specify later. Intuitively, we add $M$ to each coefficient
	of the Gomory cut in order to force the maximizer within $Q$ to be on the face of points of form $(x,\mathbf 1-x)$.
	For sufficiently large $M$ we have that
	\begin{align*}
		c^\intercal y &= (c + M\cdot \mathbf 1)^\intercal y + (M\cdot \mathbf 1)^\intercal (\mathbf 1 - y) - nM \\
		&\le\lfloor \max\{(c + M\cdot \mathbf 1)^\intercal y' + (M\cdot \mathbf 1)^\intercal z' \mid (y', z') \in Q\} \rfloor - nM \\
		&=\lfloor \max\{c^\intercal y' \mid (y', \mathbf 1 - y') \in Q\} \rfloor \\
		&= \lfloor \max\{c^\intercal x \mid x \in P\} \rfloor = \delta .
	\end{align*}
	Here, the inequality follows from $(y, \mathbf 1 - y)\in Q'$.
	Consider now the second equality.
	Since points of the form $(x, \mathbf 1 - x)$
	maximize $(M\cdot \mathbf 1)^\intercal y' + (M\cdot \mathbf 1)^\intercal z'$ with value $nM$, this will
	dominate the effects of $c$ for sufficiently large $M$.
\end{proof}

\end{document}